\documentclass{amsart}

\usepackage[left=25mm,right=25mm,top=25mm,bottom=25mm]{geometry}
\usepackage{graphicx}
\usepackage{hyperref}
\usepackage{mathpazo}
\usepackage{microtype}
\usepackage{multicol}

\theoremstyle{plain}
\newtheorem{theorem}{Theorem}
\newtheorem{lemma}[theorem]{Lemma}
\newtheorem{proposition}[theorem]{Proposition}

\theoremstyle{definition}
\newtheorem{example}[theorem]{Example}

\newcommand{\LL}{{\mathbf L}}
\newcommand{\M}{\overline{\mathcal M}}
\newcommand{\T}{{\mathcal T}}
\newcommand{\x}{{\mathbf x}}

\begin{document}

\title{The asymptotic {W}eil--{P}etersson form and intersection theory on $\M_{g,n}$}
\author{Norman Do}
\address{Department of Mathematics and Statistics, McGill University, Canada}
\email{normdo@gmail.com}
\date{\today}
\begin{abstract}
Moduli spaces of hyperbolic surfaces with geodesic boundary components of fixed lengths may be endowed with a symplectic structure via the Weil--Petersson form. We show that, as the boundary lengths are sent to infinity, the Weil--Petersson form converges to a piecewise linear form first defined by Kontsevich. The proof rests on the observation that a hyperbolic surface with large boundary lengths resembles a graph after appropriately scaling the hyperbolic metric. We also include some applications to intersection theory on moduli spaces of curves.
\end{abstract}

\maketitle

\section{Introduction} \label{introduction}

Let ${\mathcal M}_{g,n}({\mathbf L})$ denote the moduli space of hyperbolic surfaces\footnote{We decree that all surfaces referred to in this paper are to be connected and oriented. In addition, we decree that all algebraic curves referred to in this paper are to be complex, connected and complete.} with genus $g$ and $n$ labelled geodesic boundary components whose lengths are prescribed by $\LL = (L_1, L_2, \ldots, L_n)$. This moduli space may be endowed with a natural symplectic structure via the Weil--Petersson form $\omega$. In this paper, we explore the asymptotic behaviour of the symplectic structure on ${\mathcal M}_{g,n}(N{\mathbf x})$ as $N$ approaches infinity, for a fixed $\x = (x_1, x_2, \ldots, x_n)$. Underlying much of our work is the observation that a hyperbolic surface with large boundary lengths resembles a graph after appropriately scaling the hyperbolic metric. In particular, one is naturally led to consider the combinatorial structure known in the literature as a ribbon graph.

A ribbon graph of type $(g,n)$ is essentially the 1-skeleton of a cell decomposition of a genus $g$ surface which has $n$ faces. We require the vertices to have degree at least three and the faces to be labelled from 1 up to $n$. A ribbon graph with a positive real number assigned to each edge is referred to as a metric ribbon graph. The metric associates to each face a perimeter, which is simply the sum of the numbers appearing around the boundary of the face.

Let ${\mathcal MRG}_{g,n}(\x)$ denote the moduli space of metric ribbon graphs of type $(g,n)$ whose perimeters are prescribed by $\x = (x_1, x_2, \ldots, x_n)$. This is naturally an orbifold and is sometimes known as the combinatorial moduli space. Its importance lies in the fact that it is homeomorphic to ${\mathcal M}_{g,n}(\x)$ via a construction due to Bowditch and Epstein \cite{bow-eps}. In one direction, this construction associates to a hyperbolic surface with boundary its spine --- in other words, the set of points which have at least two equal shortest paths to the boundary. The spine is topologically an embedded graph which can be interpreted as a metric ribbon graph. The inverse of this construction produces a hyperbolic surface $S(\Gamma) \in {\mathcal M}_{g,n}(\x)$ for every metric ribbon graph $\Gamma \in {\mathcal MRG}_{g,n}(\x)$. Given a metric space $X$, let $\lambda X$ denote the same underlying set with the metric scaled by a positive real number $\lambda$. In this paper, we prove the following result, which formalises our earlier observation concerning hyperbolic surfaces with large boundary lengths.

\begin{theorem} \label{gromov-hausdorff}
In the Gromov--Hausdorff topology, the following equation holds for every metric ribbon graph $\Gamma$.
\[
\lim_{N \to \infty} \frac{1}{N} S(N\Gamma)= \Gamma
\]
\end{theorem}

In order to study the asymptotic behaviour of the Weil--Petersson form, fix $\x$ and consider the map
\[
f: \mathcal{MRG}_{g,n}(\x) \to \mathcal{MRG}_{g,n}(N\x) \to {\mathcal M}_{g,n}(N\x).
\]
This homeomorphism of orbifolds is the composition of two maps --- the first scales the ribbon graph metric by $N$ while the second uses the Bowditch--Epstein construction. The normalised Weil--Petersson form $\frac{\omega}{N^2}$ on ${\mathcal M}_{g,n}(N\x)$ pulls back via $f$ to a symplectic form on the combinatorial moduli space. By analysing the hyperbolic geometry of surfaces with large boundary lengths, we obtain the following result.

\begin{theorem} \label{equal-forms}
In the $N \to \infty$ limit, the symplectic form $\frac{f^*\omega}{N^2}$ converges pointwise to a piecewise linear 2-form $\Omega_L$ on the locus of trivalent metric ribbon graphs in ${\mathcal MRG}_{g,n}(\x)$. Furthermore, this coincides with the piecewise linear 2-form $\Omega_K$ introduced by Kontsevich in his proof of the Witten--Kontsevich theorem \cite{kon}.
\end{theorem}

Using a result of Mirzakhani \cite{mir2} which relates the Weil--Petersson form on ${\mathcal M}_{g,n}(\LL)$ to characteristic classes on the moduli space of curves $\M_{g,n}$ in conjunction with Theorem~\ref{equal-forms} allows us to give a new proof of the following identity.

\begin{theorem}[Kontsevich's combinatorial formula] \label{kcf} For the moduli space of curves $\M_{g,n}$, we have the following equality of rational polynomials in $s_1, s_2, \ldots, s_n$.
\[
\sum_{|\text{\boldmath$\alpha$}| = 3g-3+n} \int_{\M_{g,n}} \psi_1^{\alpha_1} \psi_2^{\alpha_2} \cdots \psi_n^{\alpha_n} ~ \prod_{k=1}^n \frac{(2\alpha_k - 1)!!}{s_k^{2\alpha_k+1}} = \sum_{\Gamma} \frac{2^{2g-2+n}}{|\text{Aut}(\Gamma)|} \prod_{e \in E(\Gamma)} \frac{1}{s_{\ell(e)} + s_{r(e)}}
\]
Here, we use the notation $|\text{\boldmath$\alpha$}|$ as a shorthand for $\alpha_1 + \alpha_2 + \cdots + \alpha_n$. The sum on the right hand side is over the trivalent ribbon graphs of type $(g,n)$. We write $\text{Aut}(\Gamma)$ and $E(\Gamma)$ for the automorphism group and edge set of $\Gamma$, respectively. For an edge $e$, the expressions $\ell(e)$ and $r(e)$ denote the labels of the faces on its left and right.\footnote{Although the left and right of an edge are not well-defined, the expression $s_{\ell(e)} + s_{r(e)}$ certainly is.}
\end{theorem}

This is the main identity used by Kontsevich in his proof of the Witten--Kontsevich theorem \cite{kon}. Our proof of this result highlights one of the goals of this paper --- namely, to bring together the combinatorial methods pioneered by Kontsevich and the hyperbolic geometry used by Mirzakhani into a coherent narrative.

As a final application, we provide a method for computing intersection numbers of the form
\[
\int_{\M_{g,n}} W \psi_1^{\alpha_1} \psi_2^{\alpha_2} \cdots \psi_n^{\alpha_n},
\]
where $W \in H^*(\M_{g,n}; \mathbb{Q})$ is the Poincar\'{e} dual to a combinatorial cycle. Combinatorial cycles, first defined by Kontsevich \cite{kon}, are represented by closures of subsets of the combinatorial moduli space corresponding to metric ribbon graphs with a specified degree sequence.

The structure of the paper is as follows.
\begin{itemize}
\item In Section~\ref{background}, we discuss the relevant background material. This includes a brief treatment of moduli spaces of hyperbolic surfaces, the combinatorial moduli space, and intersection theory on moduli spaces of curves.

\item In Section~\ref{long-boundaries}, we prove several lemmas concerning hyperbolic surfaces with large boundary lengths, culminating in the proof of Theorem~\ref{gromov-hausdorff}.

\item In Section~\ref{asymptotic-form}, we consider the asymptotic Weil--Petersson form and give the proof of Theorem~\ref{equal-forms}.

\item In Section~\ref{applications}, we provide applications of the asymptotic Weil--Petersson form to intersection theory on moduli spaces of curves. In particular, we prove Theorem~\ref{kcf} and conclude with a discussion of the method for computing intersection numbers which involve combinatorial cycles.
\end{itemize}

\section{Background} \label{background}

\subsection{Moduli spaces of hyperbolic surfaces}

For a tuple $\LL = (L_1, L_2, \ldots, L_n)$ of positive real numbers, we consider the following moduli space.
\[
{\mathcal M}_{g,n}(\LL) = \left. \left\{ (S, \beta_1, \beta_2, \ldots, \beta_n) \; \left| \begin{array}{l} S \text{ is a hyperbolic surface with genus } g \text{ and geodesic boundary} \\ \text{components } \beta_1, \beta_2, \ldots, \beta_n \text{ of lengths } L_1, L_2, \ldots, L_n \end{array} \right. \right\} \right / \sim
\]
Here, $(S, \beta_1, \beta_2, \ldots, \beta_n) \sim (T, \gamma_1, \gamma_2, \ldots, \gamma_n)$ if and only if there exists an isometry from $S$ to $T$ which sends $\beta_k$ to $\gamma_k$ for all $k$. One may extend this definition to the case $L_k = 0$ by allowing $\beta_k$ to be a hyperbolic cusp.

The methods of Teichm\"{u}ller theory enable us to construct the moduli space ${\mathcal M}_{g,n}(\LL)$ as an orbifold and to endow it with a natural symplectic structure. We start by fixing a smooth surface $\Sigma_{g,n}$ with genus $g$ and $n$ boundary components labelled from 1 up to $n$, where the Euler characteristic $\chi(\Sigma_{g,n}) = 2-2g-n$ is negative. Now define a marked hyperbolic surface to be a pair $(S, f)$ where $S$ is a hyperbolic surface and $f: \Sigma_{g,n} \to S$ is a diffeomorphism. We call $f$ the marking and define Teichm\"{u}ller space as follows.
\[
{\mathcal T}_{g,n}(\LL) = \left. \left\{ (S, f) \; \left| \begin{array}{l} (S, f) \text{ is a marked hyperbolic surface with genus } g \text{ and } \\ \text{geodesic boundary components of lengths } L_1, L_2, \ldots, L_n \end{array} \right. \right\} \right / \sim
\]
Here, $(S, f) \sim (T, g)$ if and only if there exists an isometry $\phi: S \to T$ such that $\phi \circ f$ is isotopic to $g$. In essence, Teichm\"{u}ller space is the space of all deformations of the hyperbolic structure on a surface.

We now define global coordinates on Teichm\"{u}ller space, known as Fenchel--Nielsen coordinates. Start by considering a pair of pants decomposition of the surface $\Sigma_{g,n}$ --- in other words, a collection of non-intersecting simple closed curves whose complement is a disjoint union of surfaces with genus 0 and 3 boundary components. For topological reasons, every pair of pants decomposition of $\Sigma_{g,n}$ must consist of precisely $3g-3+n$ curves. The marking $f: \Sigma_{g,n} \to S$ maps a pair of pants decomposition to a collection of simple closed curves, each of which has a unique geodesic representative in its homotopy class. Denote these simple closed geodesics by $\gamma_1, \gamma_2, \ldots, \gamma_{3g-3+n}$ and let their lengths be $\ell_1, \ell_2, \ldots, \ell_{3g-3+n}$, respectively. Cutting $S$ along $\gamma_1, \gamma_2, \ldots, \gamma_{3g-3+n}$ leaves a disjoint union of hyperbolic pairs of pants. Due to the following basic result, the length parameters $\ell_1, \ell_2, \ldots, \ell_{3g-3+n}$ provide sufficient information to reconstruct the hyperbolic structure on each pair of pants in the decomposition.

\begin{proposition} \label{pair-of-pants}
Given any three lengths, there exists a unique hyperbolic pair of pants up to isometry with geodesic boundary components of those lengths. We refer to the three simple geodesic arcs perpendicular to the boundary components and connecting them pairwise as the seams. Every hyperbolic pair of pants can be decomposed into two congruent right-angled hexagons by cutting along the seams.
\end{proposition}

On the other hand, the lengths $\ell_1, \ell_2, \ldots, \ell_{3g-3+n}$ are not sufficient to reconstruct the hyperbolic structure on all of $S$, since there are infinitely many ways to glue together the pairs of pants. This extra gluing information is stored in the twist parameters, which we denote by $\tau_1, \tau_2, \ldots, \tau_{3g-3+n}$. To construct them, fix a collection $C$ of disjoint simple curves on $\Sigma_{g,n}$ which are either closed or have endpoints on the boundary. We require that $C$ meets the curves $\gamma_1, \gamma_2, \ldots, \gamma_{3g-3+n}$ transversely, such that its restriction to any particular pair of pants consists of three disjoint arcs which connect the boundary components pairwise. Now to construct the twist parameter $\tau_k$, note that there could be either one or two curves $\gamma \in C$ such that $f(\gamma)$ meets $\gamma_k$, where $f$ is the marking. Homotopic to $f(\gamma)$, relative to the boundary of $X$, is a unique length-minimising piecewise geodesic curve which is entirely contained in the seams of the hyperbolic pairs of pants and the curves $\gamma_1, \gamma_2, \ldots, \gamma_{3g-3+n}$. The twist parameter $\tau_k$ is the signed hyperbolic distance that this curve travels along $\gamma_k$ and is well-defined due to Proposition~\ref{pair-of-pants}.

\begin{proposition} \label{fenchel-nielsen}
The map which assigns to a marked hyperbolic surface its Fenchel--Nielsen coordinates --- in other words, its length and twist parameters --- is a homeomorphism.
\[
{\mathcal T}_{g,n}(\LL) \cong \mathbb{R}_+^{3g-3+n} \times \mathbb{R}^{3g-3+n}
\]
\end{proposition}

There is clearly a projection map ${\mathcal T}_{g,n}(\LL) \to {\mathcal M}_{g,n}(\LL)$ given by forgetting the marking. In fact, we obtain the moduli space as a quotient of Teichm\"{u}ller space by the action of the mapping class group
\[
\text{Mod}_{g,n} = \text{Diff}_+(\Sigma_{g,n})/\text{Diff}_0(\Sigma_{g,n}).
\]
Here, $\text{Diff}_+$ is the group of orientation-preserving diffeomorphisms fixing each boundary component and $\text{Diff}_0$ is the normal subgroup consisting of those diffeomorphisms isotopic to the identity. There is a natural action of the mapping class group on Teichm\"{u}ller space described as follows --- if $[\phi]$ is an element of $\text{Mod}_{g,n}$, then $[\phi]$ sends the marked hyperbolic surface $(X, f)$ to the marked hyperbolic surface $(X, f \circ \phi)$. It turns out that the action of $\text{Mod}_{g,n}$ on ${\mathcal T}_{g,n}(\LL)$ is properly discontinuous, though not necessarily free. Thus, the quotient ${\mathcal M}_{g,n}(\LL) \cong \left. \T_{g,n}(\LL) \right/ \text{Mod}_{g,n}$ is naturally an orbifold.

The Teichm\"{u}ller space $\T_{g,n}(\LL)$ can be endowed with the canonical symplectic form
\[
\omega = \sum_{k=1}^{3g-3+n} d\ell_k \wedge d\tau_k
\]
using the Fenchel--Nielsen coordinates. Although this is a rather trivial statement, a remarkable fact is that this construction is invariant under the action of the mapping class group. Therefore, $\omega$ descends to a symplectic form on the quotient, namely the moduli space ${\mathcal M}_{g,n}(\LL)$. This is referred to as the Weil--Petersson form and we will also denote it by $\omega$. Its existence allows for the techniques of symplectic geometry to be used in the study of moduli spaces. It is important to note that as $\LL$ varies, the symplectic structure of ${\mathcal M}_{g,n}(\LL)$ varies, even though its smooth structure does not.

\subsection{Combinatorial moduli space}

An important notion in the study of moduli spaces is the combinatorial structure known in the literature as a ribbon graph or fatgraph. Earlier, we stated that a ribbon graph of type $(g,n)$ is essentially the 1-skeleton of a cell decomposition of a genus $g$ surface which has $n$ faces. We require the vertices to have degree at least three and the faces to be labelled from 1 up to $n$. Note that such a graph may possibly have loops or multiple edges. The orientation of the surface produces a cyclic ordering of the oriented edges pointing towards each vertex. Conversely, given the underlying graph and the cyclic ordering of the oriented edges pointing towards each vertex, the genus of the surface and its cell decomposition may be recovered. This is accomplished by using the extra structure to thicken the graph into a surface with boundaries. These boundaries may then be filled in with disks to produce a closed surface.

One usually draws ribbon graphs with the convention that the cyclic ordering of the oriented edges pointing towards each vertex is induced by the orientation of the page. For example, the following diagram shows a ribbon graph of type $(1,1)$ as well as the surface obtained by thickening the graph.

\begin{center}
\includegraphics[scale=1.1]{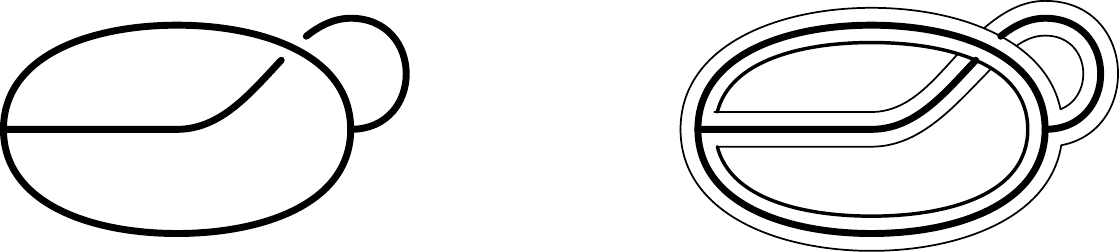}
\end{center}

Occasionally, it is useful to think of a ribbon graph in the following more precise way. Given a 1-skeleton $\Gamma$ of a cell decomposition, let $X$ denote the set of its oriented edges and let $s_0$ be the permutation on $X$ which cyclically permutes all oriented edges pointing towards the same vertex in an anticlockwise manner. Also, let $s_1$ be the permutation on $X$ which interchanges each pair of oriented edges which correspond to the same underlying edge. The set $X / \langle s_0 \rangle$ is canonically equivalent to the set of vertices of $\Gamma$ while the set $X / \langle s_1 \rangle$ is canonically equivalent to the set of edges of $\Gamma$. Furthermore, if we let $s_2 = s_0^{-1} s_1$, then the set $X / \langle s_2 \rangle$ is canonically equivalent to the set of faces of $\Gamma$. Therefore, one can alternatively consider a ribbon graph to be a triple $(X, s_0, s_1)$ where $X$ is a finite set, $s_0$ is a permutation on $X$ without fixed points or transpositions and $s_1$ is an involution on $X$ without fixed points. We also require a labelling in the form of a bijection from $X / \langle s_2 \rangle$ to $\{1, 2, \ldots, n\}$. Define two ribbon graphs $(X, s_0, s_1)$ and $(\overline{X}, \overline{s}_0, \overline{s}_1)$ to be isomorphic if and only if there exists a bijection $f: X \to \overline{X}$ such that $f \circ s_0 = \overline{s}_0 \circ f$ and $f \circ s_1 = \overline{s}_1 \circ f$. We also impose the condition that $f$ must preserve the labelling of the boundary components. A ribbon graph automorphism is, of course, an isomorphism from a ribbon graph to itself. The set of automorphisms of a ribbon graph $\Gamma$ forms a group which is denoted by $\text{Aut}(\Gamma)$.

As mentioned earlier, a ribbon graph with a positive real number assigned to each edge is referred to as a metric ribbon graph. The metric associates to each face in the cell decomposition a perimeter, which is simply the sum of the numbers appearing around the boundary of the face. Let ${\mathcal MRG}_{g,n}(\x)$ denote the moduli space of metric ribbon graphs of type $(g,n)$ whose perimeters are prescribed by $\x = (x_1, x_2, \ldots, x_n)$, where equivalence of metric ribbon graphs corresponds to isometry of metric spaces. For every ribbon graph $\Gamma$ of type $(g, n)$, consider the subset ${\mathcal MRG}_\Gamma(\x) \subseteq {\mathcal MRG}_{g,n}(\x)$ consisting of those metric ribbon graphs whose underlying ribbon graph is $\Gamma$. Note that ${\mathcal MRG}_\Gamma(\x)$ can be naturally identified with the following quotient of a possibly empty polytope by a finite group.
\[
\left. \left\{ \mathbf{e} \in \mathbb{R}_+^{E(\Gamma)} \; \middle \vert \; A_\Gamma \mathbf{e} = \x \right\} \right/ \text{Aut}(\Gamma)
\]
Here, $\mathbf{e}$ represents the lengths of the edges in the metric ribbon graph while $A_\Gamma$ is the linear map which represents the adjacency between faces and edges in the cell decomposition corresponding to $\Gamma$. Thus, ${\mathcal MRG}_\Gamma(\x)$ is an orbifold cell and these naturally glue together via edge degenerations --- in other words, when an edge length goes to zero, the edge contracts to give a ribbon graph with fewer edges. This cell decomposition for ${\mathcal MRG}_{g,n}(\x)$ equips it with not only a topology, but also an orbifold structure. The main reason for considering ${\mathcal M}_{g,n}(\x)$, sometimes known as the combinatorial moduli space, is the following.

\begin{proposition} \label{bowditch-epstein}
The moduli spaces ${\mathcal M}_{g,n}(\x)$ and ${\mathcal MRG}_{g,n}(\x)$ are homeomorphic as orbifolds.
\end{proposition}

This equivalence can be shown via uniformisation and the notion of Jenkins--Strebel quadratic differentials, as originally noted by Harer, Mumford and Thurston. However, it will be more advantageous for us to consider the hyperbolic geometric proof due to Bowditch and Epstein \cite{bow-eps}. The main idea is to associate to a hyperbolic surface $S$ with geodesic boundary its spine $\Gamma(S)$. For every point $p \in S$, let $n(p)$ denote the number of shortest paths from $p$ to the boundary. Generically, we have $n(p) = 1$ and we define the spine as
\[
\Gamma(S) = \{p \in S \mid n(p) \geq 2\}.
\]
The locus of points with $n(p) = 2$ consists of a disjoint union of open geodesic segments. These correspond precisely to the edges of a graph embedded in $S$. The locus of points with $n(p) \geq 3$ forms a finite set which corresponds to the set of vertices of the aforementioned graph. In fact, if $n(p) \geq 3$, then the corresponding vertex will have degree $n(p)$. In this way, $\Gamma(S)$ has the structure of a ribbon graph. Furthermore, it is a deformation retract of the original hyperbolic surface, so if $S$ has genus $g$ and $n$ boundary components, then $\Gamma(S)$ will be a ribbon graph of type $(g,n)$.

Now for each vertex $p$ of $\Gamma(S)$, consider the $n(p)$ shortest paths from $p$ to the boundary. We refer to these geodesic segments as ribs and note that they are perpendicular to the boundary of $S$. The diagram below shows part of a hyperbolic surface, along with its spine and ribs. Cutting $S$ along its ribs leaves a collection of hexagons, each with four right angles and a reflective axis of symmetry along one of the diagonals. In fact, this diagonal is one of the edges of $\Gamma(S)$ and we assign to it the length of the side of the hexagon which lies along the boundary of $S$. Of course, there are two such sides --- however, the reflective symmetry guarantees that they are equal in length. In this way, $\Gamma(S)$ becomes a metric ribbon graph of type $(g,n)$. By construction, the perimeters of $\Gamma(S)$ correspond precisely with the lengths of the boundary components of $S$, so we have a map $\Gamma: {\mathcal M}_{g,n}(\x) \to {\mathcal MRG}_{g,n}(\x)$. Although Bowditch and Epstein considered only the case of cusped hyperbolic surfaces, one can show that this is in fact a homeomorphism of orbifolds using a proof entirely analogous to theirs. Further details may be found in \cite{bow-eps, do}.

\begin{center}
\includegraphics{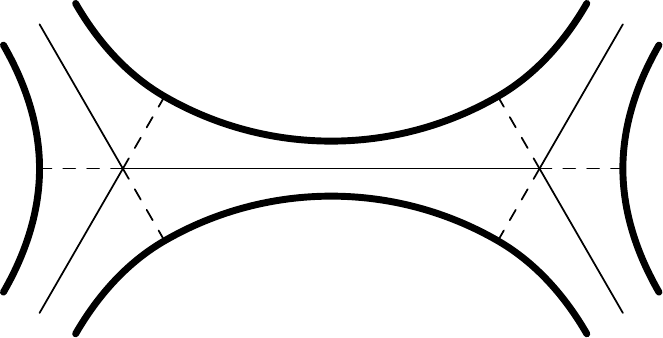}
\end{center}

\subsection{Intersection theory on moduli spaces of curves}

For non-negative integers $g$ and $n$ satisfying the Euler characteristic condition $2 - 2g - n < 0$, define the moduli space of curves as follows.
\[
{\mathcal M}_{g,n} = \left. \left\{ (C, p_1, p_2, \ldots, p_n) \; \left| \begin{array}{l} C \text{ is a smooth algebraic curve of genus } g \\ \text{with } n \text{ distinct points } p_1, p_2, \ldots, p_n \end{array} \right. \right\} \right / \sim
\]
Here, $(C, p_1, p_2, \ldots, p_n) \sim (D, q_1, q_2, \ldots, q_n)$ if and only if there exists an isomorphism from $C$ to $D$ which sends $p_k$ to $q_k$ for all $k$. Of course, one can equivalently consider these algebraic curves as Riemann surfaces with punctures rather than marked points. The uniformisation theorem then allows us to deduce that ${\mathcal M}_{g,n}$, with its natural topology, is diffeomorphic to the moduli space of hyperbolic surfaces ${\mathcal M}_{g,n}({\mathbf L})$ for all $\mathbf L$.

It is often more natural to work with the Deligne--Mumford compactification of the moduli space of curves.
\[
\overline{\mathcal M}_{g,n} = \left. \left\{ (C, p_1, p_2, \ldots, p_n) \; \left| \begin{array}{l} C \text{ is a stable algebraic curve of genus } g \text{ with} \\ n \text{ distinct smooth points } p_1, p_2, \ldots, p_n \end{array} \right. \right\} \right / \sim
\]
An algebraic curve is called stable if it has at worst nodal singularities and a finite automorphism group. The practical interpretation of this latter condition is that every rational component of the curve must have at least three points which are nodes or marked points. An important problem concerning $\M_{g,n}$ is the calculation of its intersection theory with respect to certain characteristic classes. The classes that we will consider live in the cohomology ring $H^*(\M_{g,n})$ and arise from taking Chern classes of natural complex vector bundles.\footnote{Readers with a more algebraic predilection may prefer to think of these classes as living in the Chow ring $A^*(\M_{g,n})$.}

Given a stable genus $g$ curve with $n+1$ marked points, one can forget the point labelled $n+1$ to obtain a genus $g$ curve with $n$ marked points. Unfortunately, the resulting curve may not be stable, but gives rise to a well-defined stable curve after contracting all unstable rational components. This gives a map $\pi: \M_{g,n+1} \to \M_{g,n}$ known as the forgetful morphism and which can be interpreted as the universal family over $\M_{g,n}$. In other words, one can take a stable curve $C \in \M_{g,n}$ along with a point $p$ on $C$ and associate a stable curve in $\M_{g,n+1}$ to the pair $(C,p)$. In particular, the fibre over a point $C \in \M_{g,n}$ is essentially the marked curve described by $C$. So the point labelled $k$ defines a section $\sigma_k: \M_{g,n} \to \M_{g,n+1}$ for $k = 1, 2, \ldots, n$. The forgetful morphism can be used to pull back cohomology classes, but it will also be useful to push them forward. This is possible via the Gysin map $\pi_*: H^*(\M_{g,n+1}) \to H^*(\M_{g,n})$, a homomorphism of graded rings with grading $-2$ which can be described as integration along fibres.

Now consider the vertical cotangent bundle on $\M_{g,n+1}$ with fibre at $(C,p)$ equal to the cotangent line $T_p^*C$. Unfortunately, this definition is nonsensical when $p$ is a singular point of $C$. Therefore, it is necessary to consider the relative dualising sheaf, the unique line bundle on $\M_{g,n+1}$ which extends the vertical cotangent bundle. More precisely, if we let ${\mathcal K}_X$ be the canonical line bundle on $\M_{g,n+1}$ and ${\mathcal K}_B$ be the canonical line bundle on $\M_{g,n}$, then it can be defined as
\[
{\mathcal L} = {\mathcal K}_X \otimes \pi^* {\mathcal K}_B^{-1}.
\]

There are natural line bundles on $\M_{g,n}$ formed by pulling back $\mathcal L$ along the sections $\sigma_k: \M_{g,n} \to \M_{g,n+1}$ for $k = 1, 2, \ldots, n$. Taking Chern classes of these line bundles, we obtain the psi-classes
\[
\psi_k = c_1(\sigma_k^* {\mathcal L}) \in H^2(\M_{g,n}; \mathbb{Q}) \qquad \text{for } k = 1, 2, \ldots, n.
\]

Define the twisted Euler class by $e = c_1 \left({\mathcal L} \left( D_1 + D_2 + \cdots + D_n \right) \right)$, where $D_k$ is the divisor on $\M_{g,n+1}$ representing the image of the section $\sigma_k$. Taking the push-forward of its powers, we obtain the kappa-classes
\[
\kappa_m = \pi_*(e^{m+1}) \in H^{2m}(\M_{g,n}; \mathbb{Q}) \qquad \text{for } m = 0, 1, 2, \ldots, 3g-3+n.
\]
In the following, we will only be dealing with $\kappa_1$, which appears naturally in the study of moduli spaces of hyperbolic surfaces via the following result of Wolpert \cite{wol2}.

\begin{proposition} \label{de-rham}
The Weil--Petersson form $\omega$ on the moduli space of hyperbolic surfaces ${\mathcal M}_{g,n}(\mathbf{0})$ induces the de~Rham cohomology class
\[
[\omega] = 2\pi^2 \kappa_1 \in H^2(\M_{g,n}; \mathbb{R}).
\]
\end{proposition}

\section{Hyperbolic surfaces with large boundary lengths} \label{long-boundaries}

Fix a metric ribbon graph $\Gamma \in {\mathcal MRG}_{g,n}(\x)$ and let $N\Gamma \in {\mathcal MRG}_{g,n}(N\x)$ denote the same underlying ribbon graph with the metric scaled by a factor of $N$. We will be interested in the geometry of the hyperbolic surface $S(N\Gamma)$ arising from the Bowditch--Epstein construction as $N$ approaches infinity. The Gauss--Bonnet theorem ensures that the surface area remains constant in the limit. So as the boundary lengths become large, the surface appears to stretch and resemble a graph with the surface area concentrated around the vertices. One goal of this section is to formalise this intuitive picture.

The Bowditch--Epstein construction produces the metric ribbon graph $N\Gamma$ embedded as the spine of the hyperbolic surface $S(N\Gamma)$. Previously, we defined a rib to be a shortest path from a vertex of the spine to the boundary of the surface. Cutting $S(N\Gamma)$ along its ribs leaves a collection of hexagons --- let us call them edge hexagons --- each of which includes a unique edge of $N\Gamma$. Given an edge hexagon, one can lift it to the hyperbolic plane and consider the two sides which are parallel to the edge. We refer to the common perpendicular between the two corresponding lines as an intercostal.

\begin{lemma} \label{intercostal}
If $\delta(e)$ denotes the length of the intercostal corresponding to an edge $e$ in $N\Gamma$, then
\[
\lim_{N \to \infty} \delta(e) = 0.
\]
\end{lemma}

\begin{proof}
The diagram below shows the edge $e$ and its corresponding edge hexagon. The boundary of the hexagon comprises four ribs which have been drawn as dotted lines and two boundary segments which have been drawn as solid lines. Note that the intercostal is perpendicular to the two boundary segments and hence, by symmetry, must also be perpendicular to the edge $e$. So there are four hyperbolic trirectangles in the diagram.

\begin{center}
\includegraphics{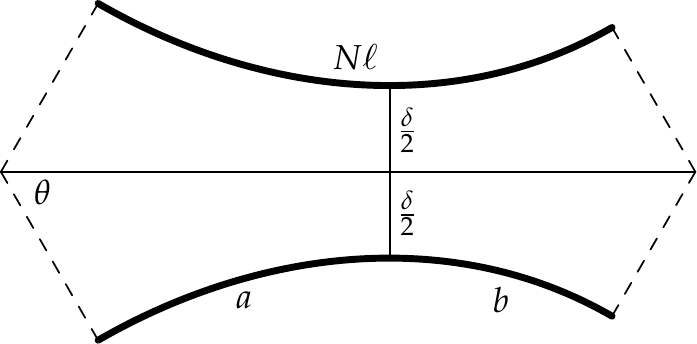}
\end{center}

Suppose that the length of the edge $e$ in the metric ribbon graph $N\Gamma$ is $N\ell$. Consider the lower left trirectangle in the diagram and, without loss of generality, assume that the length marked $a$ satisfies $a \geq \frac{N\ell}{2}$. By the standard trigonometric formula for trirectangles --- for example, consider the reference \cite{bus} --- we have the equation $\cos \theta = \sinh a \sinh \frac{\delta}{2}$. Therefore,
\[
0 \leq \sinh \frac{\delta}{2} = \frac{\cos \theta}{\sinh a} \leq \frac{1}{\sinh \frac{N \ell}{2}}
\]
and taking the $N \to \infty$ limit leads to the desired result.
\end{proof}

Although the proof of the previous lemma remains valid, the intercostal might not actually intersect the edge, as depicted in the diagram. However, the next lemma guarantees that this assumption is indeed correct, at least for $N$ sufficiently large.

\begin{lemma} \label{intercostal-edge}
If $N$ is sufficiently large, then the intercostal corresponding to an edge intersects the edge.
\end{lemma}

\begin{proof}
To obtain a contradiction, suppose that the intercostal corresponding to the edge $e$ does not intersect it. Consider a lift of $e$ to the hyperbolic plane, along with the corresponding intercostal of length $\delta$, the adjacent boundary components, and the ribs joining $e$ to these boundary components. Let the vertex $v$ closer to the intercostal be at distance $r$ from the two adjacent boundary components.

\begin{center}
\includegraphics{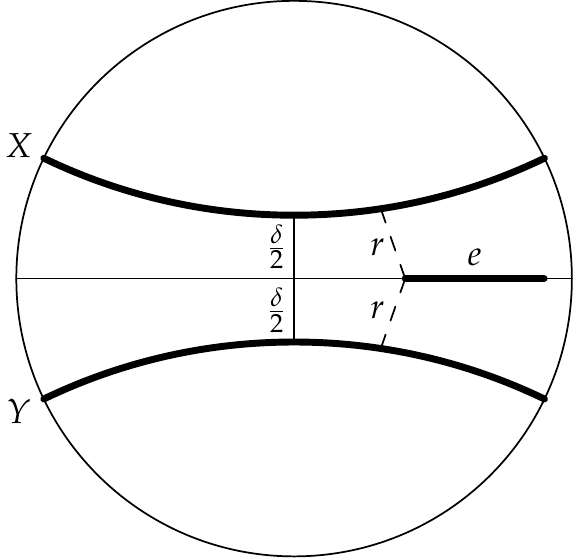}
\end{center}

There must exist a third lift of a boundary component which is also at distance $r$ from the vertex $v$. Furthermore, the endpoints of this lift must lie between the points labelled $X$ and $Y$ in the diagram. By Lemma~\ref{intercostal}, the length $\delta$ converges to zero as $N$ approaches infinity. And as $\delta$ converges to zero, it is clear that $v$ cannot remain equidistant from the three lifts of boundaries without the intercostal intersecting the edge.
\end{proof}

Since we are interested in the $N \to \infty$ limit, we may assume herein that the intercostal corresponding to an edge is the common perpendicular between opposite sides in the edge hexagon. Cutting $S(N\Gamma)$ along its intercostals leaves a collection of right-angled polygons --- let us call them vertex polygons --- each of which includes a unique vertex of $\Gamma$. Note that, at a vertex of degree $d$ in $N\Gamma$, the corresponding vertex polygon has $2d$ sides.

\begin{lemma} \label{angle-convergence}
In the $N \to \infty$ limit, the angles between adjacent edges at a degree $d$ vertex converge to $\frac{2\pi}{d}$.
\end{lemma}

\begin{proof}
Consider a lift of a vertex polygon to the hyperbolic plane. If the vertex has degree $d$ in $N\Gamma$, then there will be $d$ edges and $d$ ribs which meet there. The vertex polygon is divided into $2d$ trirectangles by these edges and ribs. The symmetry in an edge hexagon implies that the two angles labelled $\alpha$ and $\beta$ in the diagram below must be equal.

\begin{center}
\includegraphics{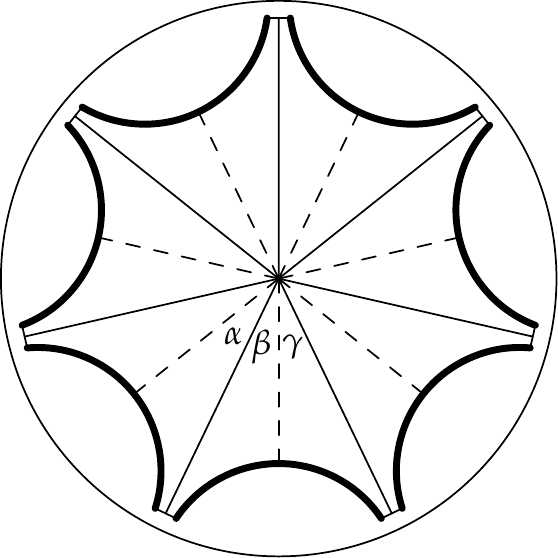}
\end{center}

It remains to show that the two angles labelled $\beta$ and $\gamma$ converge to the same value. However, in the $N \to \infty$ limit, Lemma~\ref{intercostal} ensures that the length of each intercostal approaches zero. Therefore the trirectangles which include the angles labelled $\beta$ and $\gamma$ converge to ideal triangles with one right-angle and one ideal vertex. Since they also share a common side, it follows from basic hyperbolic trigonometry that $\beta$ and $\gamma$ converge to the same value.
\end{proof}

We have deduced that in the $N \to \infty$ limit, a vertex polygon corresponding to a degree $d$ vertex resembles a regular ideal $d$-gon. One consequence is that the lengths of the ribs in $S(N\Gamma)$ at a degree $d$ vertex converge to the finite value $\cosh^{-1}\left( \frac{1}{\sin \pi/d} \right)$.

\begin{lemma} \label{length-lemma}
Let $\gamma$ be a closed geodesic of length $\ell(\gamma)$ in $\Gamma$. Since $N\Gamma$ is a deformation retract of the hyperbolic surface $S(N\Gamma)$, this defines a unique closed geodesic on $S(N\Gamma)$. After scaling the hyperbolic metric by $\frac{1}{N}$, we obtain a closed geodesic $\gamma_N$ on $\frac{1}{N} S(N\Gamma)$ whose length we denote by $\ell(\gamma_N)$. Then we have the equation
\[
\lim_{N \to \infty} \ell(\gamma_N)= \ell(\gamma).
\]
\end{lemma}

\begin{proof}
If $\gamma$ travels along an edge of $\Gamma$, then the geodesic $\gamma_N$ must travel through the corresponding edge hexagon in $\frac{1}{N}S(N\Gamma)$. The segment of $\gamma_N$ which travels through this edge hexagon is at least as long as the corresponding edge in $\Gamma$. One can see this by projecting the segment onto the appropriate side of the hexagon. Summing up over the edges traversed by $\gamma$, we obtain the fact that
\[
\ell(\gamma_N) \geq \ell(\gamma).
\]

On the other hand, if we consider $\Gamma$ embedded as the spine of $\frac{1}{N}S(N\Gamma)$, then the curve $\gamma$ in the hyperbolic surface consists of geodesic segments along the edges of $\Gamma$. By the triangle inequality, the length of such a segment in $\frac{1}{N}S(N\Gamma)$ exceeds the length of the corresponding edge in $\Gamma$ by no more than twice the maximum rib length in $\frac{1}{N}S(N\Gamma)$. Summing up over the edges traversed by $\gamma$, we obtain the fact that
\[
\ell(\gamma_N) \leq \ell(\gamma) + \frac{2Er}{N}.
\]
Here, $E$ is the number of edges traversed by $\gamma$ and $r$ is the maximum rib length in $S(N\Gamma)$. However, from the observation previous to this lemma, $r$ converges to the finite value $\cosh^{-1}\left( \frac{1}{\sin \pi/d} \right)$, where $d$ is the maximum degree of a vertex in $\Gamma$. One then obtains the desired result by taking the $N \to \infty$ limit of the two-sided inequality
\[
\ell(\gamma) \leq \ell(\gamma_N) \leq \ell(\gamma) + \frac{2Er}{N}. \qedhere
\]
\end{proof}

In Section~\ref{asymptotic-form}, our focus will be on trivalent metric ribbon graphs --- in other words, those whose vertices all have degree three. For such a metric ribbon graph $\Gamma$, the angles between closed geodesics in $S(N\Gamma)$ are particularly simple to compute in the $N \to \infty$ limit.

\begin{lemma} \label{angle-lemma}
Let $\Gamma$ be a trivalent metric ribbon graph and consider two distinct closed geodesics in $\Gamma$. Suppose that the corresponding closed geodesics in $S(N\Gamma)$ intersect in an angle $\theta \leq \frac{\pi}{2}$. Then we have the equation
\[
\lim_{N \to \infty} \theta = 0.
\]
\end{lemma}

\begin{proof}
Since $\Gamma$ is trivalent, each vertex polygon in $S(N\Gamma)$ is a right-angled hexagon. These hexagons have three alternating sides which are intercostals and three alternating sides which are boundary segments. If a closed geodesic enters a vertex polygon via one intercostal, it must exit via another. Therefore, the intersection of two closed geodesics must resemble one of the following two diagrams.

\begin{center}
\includegraphics{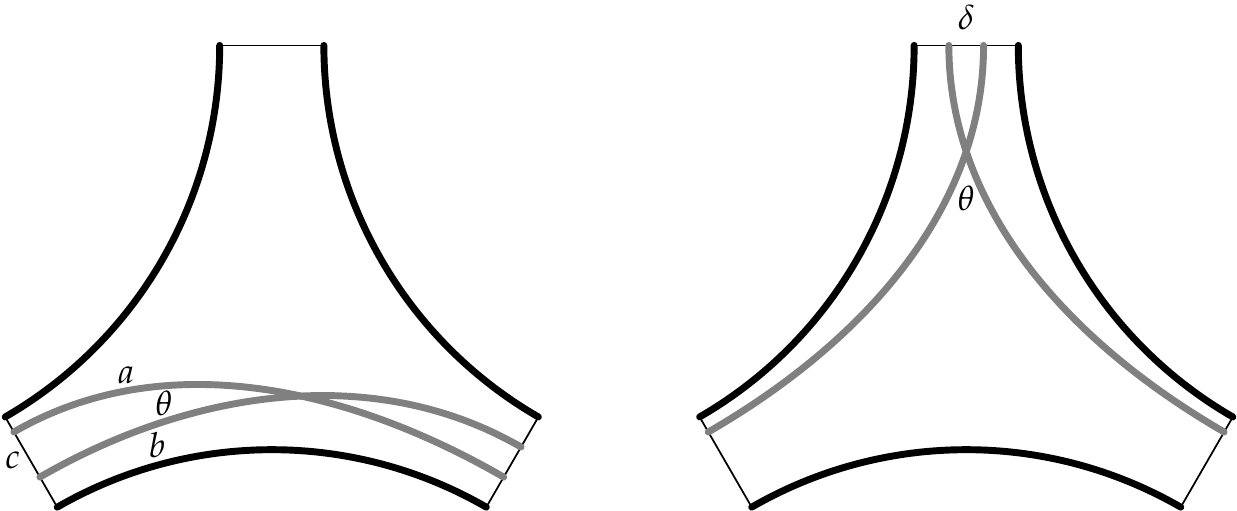}
\end{center}

In the diagram on the left, as $N$ approaches infinity, we may assume without loss of generality that the lengths denoted $a$ and $b$ also approach infinity. Since we know by Lemma~\ref{intercostal} that the length of an intercostal must converge to 0, so must the length denoted by $c$ in the diagram. The hyperbolic cosine rule states that $\cosh c = \cosh a \cosh b - \sinh a \sinh b \cos \theta$ or equivalently,
\[
\cos \theta = \coth a \coth b - \frac{\cosh c}{\sinh a \sinh b}.
\]
So in the $N \to \infty$ limit, we know that $\cos \theta$ converges to 1 and $\theta$ converges to 0.

In the diagram on the right, as $N$ approaches infinity, Lemma~\ref{intercostal} asserts that the length denoted by $\delta$ in the diagram converges to 0. Since the hexagon is right-angled, the two boundary segments adjacent to this intercostal limit to an ideal vertex. So in the $N \to \infty$ limit, we know that $\theta$ converges to 0.
\end{proof}

These results suggest that hyperbolic surfaces with large boundary lengths resemble metric ribbon graphs after appropriately scaling the hyperbolic metric. A precise statement of this fact can be made by making use of Gromov--Hausdorff convergence. Given a metric space $X$, let $\lambda X$ denote the same underlying set with the metric scaled by a positive real number $\lambda$. Theorem~\ref{gromov-hausdorff} states that, in the Gromov--Hausdorff topology, the following equation holds for every metric ribbon graph $\Gamma$.
\[
\lim_{N \to \infty} \frac{1}{N} S(N\Gamma)= \Gamma
\]

\begin{proof}[Proof of Theorem~\ref{gromov-hausdorff}]
We will require the notion of an $\epsilon$-GHA or, in other words, an $\epsilon$-Gromov--Hausdorff approximation. We say that a map $f: X \to Y$ between metric spaces is an $\epsilon$-GHA if $Y$ is contained in the $\epsilon$-neighbourhood of $f(X)$ and
\[
|d_X(x_1, x_2) - d_Y(f(x_1), f(x_2))| < \epsilon.
\]
This concept can be used to define a metric on the set of metric spaces which is well-known to be equivalent to the Gromov--Hausdorff metric.
\[
d(X, Y) = \inf~\{\epsilon > 0 \mid \text{there exist $\epsilon$-GHAs } f: X \to Y \text{ and } g: Y \to X\}
\]

Define the map $f: \Gamma \to \frac{1}{N} S(N\Gamma)$ by inclusion as the spine. From the same length estimates used in the proof of Lemma~\ref{length-lemma}, we obtain
\[
d_S(f(x_1), f(x_2)) \leq d_\Gamma(x_1, x_2) \leq d_S(f(x_1), f(x_2)) + \frac{2Er}{N},
\]
where $E$ is the number of edges traversed by the geodesic from $x_1$ to $x_2$ and $r$ is the maximum rib length in $S(N\Gamma)$. As observed earlier, $r$ converges to a finite value while $E$ is clearly bounded above by the number of edges in $\Gamma$. In addition, we know that the $\frac{r}{N}$-neighbourhood of $f(\Gamma)$ contains $\frac{1}{N}S(N\Gamma)$. Therefore, $f: \Gamma \to \frac{1}{N} S(N\Gamma)$ is an $\epsilon$-GHA where $\epsilon$ converges to 0 in the $N \to \infty$ limit.

Now define the map $g: \frac{1}{N} S(N\Gamma) \to \Gamma$ in the following way. For a point $x \in \frac{1}{N} S(N\Gamma)$, extend the shortest path from the boundary to $x$ until it meets the spine $\Gamma$ at $g(x)$. From the same length estimates used in the proof of Lemma~\ref{length-lemma}, we obtain
\[
d_\Gamma(g(x_1), g(x_2)) \leq d_S(x_1, x_2) \leq d_\Gamma(g(x_1), g(x_2)) + \frac{2Er}{N},
\]
where $E$ is the number of edges traversed by the image of the the geodesic from $x_1$ to $x_2$ and $r$ is the maximum rib length in $S(N\Gamma)$. Once again, $r$ converges to a finite value while $E$ is clearly bounded above by the number of edges in $\Gamma$. In addition, we know that the image of $g$ is precisely $\Gamma$. Therefore, $g: \frac{1}{N} S(N\Gamma) \to \Gamma$ is an $\epsilon$-GHA where $\epsilon$ converges to 0 in the $N \to \infty$ limit. It follows that $\frac{1}{N} S(N\Gamma)$ converges to $\Gamma$ in the Gromov--Hausdorff topology.
\end{proof}

\section{The asymptotic Weil--Petersson form} \label{asymptotic-form}

In order to study the asymptotic behaviour of the Weil--Petersson form, fix $\x$ and consider the map
\[
f: \mathcal{MRG}_{g,n}(\x) \to \mathcal{MRG}_{g,n}(N\x) \to {\mathcal M}_{g,n}(N\x).
\]
This homeomorphism of orbifolds is the composition of two maps --- the first scales the ribbon graph metric by $N$ while the second uses the Bowditch--Epstein construction. The normalised Weil--Petersson form $\frac{\omega}{N^2}$ on ${\mathcal M}_{g,n}(N\x)$ pulls back via $f$ to a symplectic form on the combinatorial moduli space. We will be interested in the asymptotic behaviour of this symplectic form.

Now fix a trivalent ribbon graph $\Gamma$ of type $(g,n)$ and consider the set ${\mathcal MRG}_\Gamma(\x) \subseteq {\mathcal MRG}_{g,n}(\x)$. This is an open orbifold cell in the cell decomposition for the combinatorial moduli space described in Section~\ref{background}. The ribbon graph $\Gamma$ has exactly $6g-6+3n$ edges, which we label from 1 up to $6g-6+3n$. The lengths of these edges $e_1, e_2, \ldots, e_{6g-6+3n}$ provide a set of natural coordinates on ${\mathcal MRG}_\Gamma(\x)$. One goal of this section is to prove that there exist constants $a_{ij}$ which depend on $\Gamma$ such that $\frac{f^*\omega}{N^2}$ converges to a 2-form on ${\mathcal MRG}_\Gamma(\x)$ of the form
\[
\lim_{N \to \infty} \frac{f^*\omega}{N^2} = \sum_{i < j} a_{ij} ~ de_i \wedge de_j.
\]

Although the Fenchel--Nielsen coordinates are canonical for the Weil--Petersson form, it is useful to instead consider local coordinates for the moduli space, each of which is the length function associated to some simple closed curve. The following result due to Wolpert \cite{wol1} asserts that the Weil--Petersson form has a reasonably straightforward description in such coordinates.

\begin{proposition} \label{wolpert}
Let $C_1, C_2, \ldots, C_{6g-6+2n}$ be distinct simple closed geodesics with lengths $\ell_1, \ell_2, \ldots, \ell_{6g-6+2n}$ in a hyperbolic surface with genus $g$ and $n$ cusps. If $C_i$ and $C_j$ meet at a point $p$, let $\theta_p$ denote the angle between the curves, measured anticlockwise from $C_i$ to $C_j$. Define the $(6g-6+2n) \times (6g-6+2n)$ skew-symmetric matrix $X$ by the formula
\[
X_{ij} = \sum_{p \in C_i \cap C_j} \cos \theta_p, \qquad \text{for } i < j.
\]
If $X$ is invertible, then $\ell_1, \ell_2, \ldots, \ell_{6g-6+2n}$ are local coordinates for the moduli space and the Weil--Petersson form is given by
\[
\omega = - \sum_{i < j}~[X^{-1}]_{ij}~d\ell_i \wedge d\ell_j.
\]
\end{proposition}

On closer inspection of Wolpert's original proof, this result extends without amendment to the case of hyperbolic surfaces with geodesic boundary. By linearity, it also holds if $C_1, C_2, \ldots, C_{6g-6+2n}$ are simple geodesic multicurves or, in other words, finite unions of disjoint simple closed geodesics, each with a positive weight. In order to use this expression for the Weil--Petersson form, we require a natural system of multicurves to work with. We associate a multicurve $\widetilde{C}_k$ in $\Gamma$ to the edge labelled $k$ using the following convention.

\begin{itemize}
\item Case 1: If the edge labelled $k$ is adjacent to the two distinct faces labelled $i$ and $j$, then let $\widetilde{C}_k$ be the curve shown in bold in the diagram below left.

\item Case 2: If the edge labelled $k$ is adjacent to the face labelled $i$ on both sides, then let $\widetilde{C}_k$ be the union of the two curves shown in bold in the diagram below right.

\item Case 3: If the edge labelled $k$ is a loop, then let $\widetilde{C}_k$ be the empty curve.
\end{itemize}

\begin{center}
\includegraphics{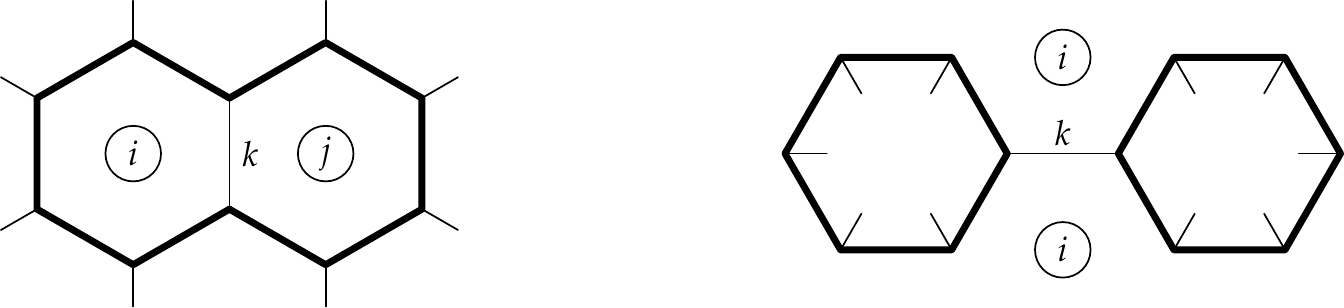}
\end{center}

Now given a trivalent metric ribbon graph $N\Gamma \in {\mathcal MRG}_\Gamma(N\x)$, the multicurve $\widetilde{C}_k$ on $\Gamma$ defines a simple geodesic multicurve $C_k$ on $S(N\Gamma)$.\footnote{We use $\Gamma$ to denote both the metric ribbon graph as well as its underlying ribbon graph --- hopefully, no confusion should arise from this abuse of notation.} If we denote its length by $\ell_k$, then we obtain length functions $\ell_1, \ell_2, \ldots, \ell_{6g-6+3n}$ which can be pulled back to ${\mathcal MRG}_\Gamma(\x)$. However, we will be more interested in the normalised length functions $\widehat{\ell}_1, \widehat{\ell}_2, \ldots, \widehat{\ell}_{6g-6+3n}$ defined by $\widehat{\ell}_k = \frac{\ell_k}{N}$. For a particular value of $N$, it is difficult to precisely relate this coordinate system with that given by the edge length functions $e_1, e_2, \ldots, e_{6g-6+3n}$. However, the picture is much simpler in the $N \to \infty$ limit, where we can use Lemma~\ref{length-lemma} to deduce the following.

\begin{itemize}
\item Case 1: If the edge labelled $k$ is adjacent to the two distinct faces labelled $i$ and $j$, then 
\[
\lim_{N \to \infty} \widehat{\ell}_k = x_i + x_j - 2e_k.
\]

\item Case 2: If the edge labelled $k$ is adjacent to the face labelled $i$ on both sides, then 
\[
\lim_{N \to \infty} \widehat{\ell}_k = x_i - 2e_k.
\]

\item Case 3: If the edge labelled $k$ is a loop, then 
\[
\lim_{N \to \infty} \widehat{\ell}_k = 0.
\]
\end{itemize}

The normalised length functions are known to be real analytic functions and, by the work of Wolpert \cite{wol3}, are also known to have bounded first and second derivatives. Therefore, we may interchange the order of limit and derivative to obtain
\[
\lim_{N \to \infty} d\widehat{\ell}_k = -2de_k, \qquad \text{ for } k = 1, 2, \ldots, 6g-6+3n.
\]

We now turn our attention to the asymptotic behaviour of the $(6g-6+3n) \times (6g-6+3n)$ skew-symmetric matrix defined by
\[
{X}_{ij} = \sum_{p \in C_i \cap C_j} \cos \theta_p, \qquad \text{for } i < j.
\]

Given the trivalent ribbon graph $\Gamma$, we define $B_{ij}$ --- the oriented adjacency between edge $i$ and edge $j$ --- to be 0 if the edges are not adjacent or equal, and according to the following convention otherwise.

\begin{center}
\includegraphics{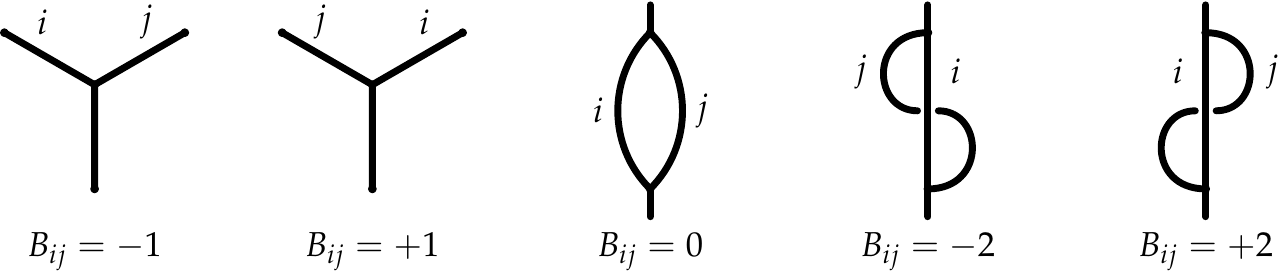}
\end{center}

This also defines a $(6g-6+3n) \times (6g-6+3n)$ skew-symmetric matrix. Note that $B$ is an integer matrix by construction while $X$ converges to an integer matrix in the $N \to \infty$ limit as a result of Lemma~\ref{angle-lemma}. In fact, these two matrices are related by the following result.

\begin{lemma} \label{matrix-convergence}
In the $N \to \infty$ limit, the matrix $X$ converges to $-2B$.
\end{lemma}

\begin{proof}
Suppose that the two curves $\widetilde{C}_i$ and $\widetilde{C}_j$ traverse a maximal path of consecutive edges in $\Gamma$. Then in the $N \to \infty$ limit, they will contribute $+1$ to $X_{ij}$ if the diagram resembles the following, $-1$ if the order of the curves is reversed, and 0 otherwise.

\begin{center}
\includegraphics{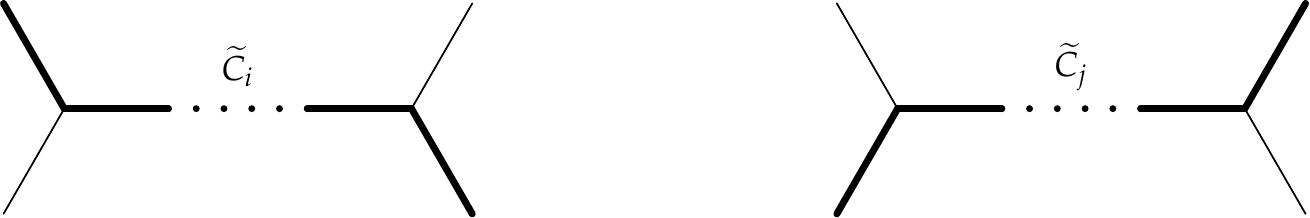}
\end{center}

It is clear that if edges $i$ and $j$ are not adjacent to a common face, then $C_i$ and $C_j$ do not intersect and $X_{ij} = 0$. Now suppose that edges $i$ and $j$ do share a common face, but are not adjacent. Then the schematic diagram below, combined with the previous observation, shows that $C_i$ and $C_j$ meet precisely twice. However, the two corresponding contributions to $X_{ij}$ have different signs, so ${X}_{ij} = 0$ in the $N \to \infty$ limit.

\begin{center}
\includegraphics{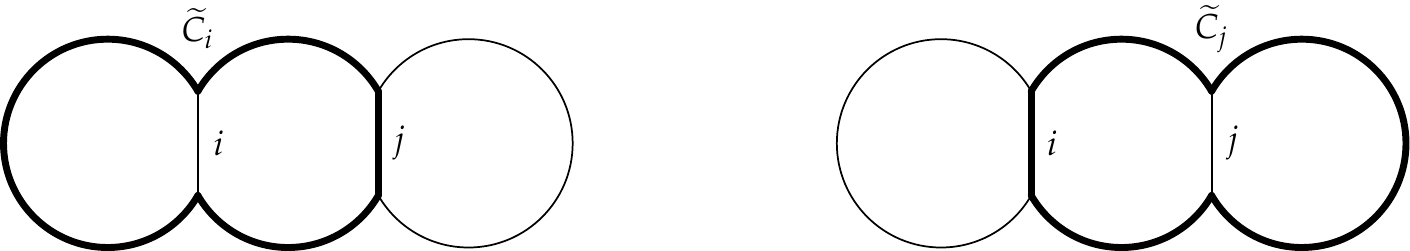}
\end{center}

Now suppose that the oriented adjacency between edges $i$ and $j$ is $-1$. Then the schematic diagram below, combined with our previous observation, shows that $C_i$ and $C_j$ meet precisely twice. The two corresponding contributions to $X_{ij}$ are positive, so ${X}_{ij} = 2$  in the $N \to \infty$ limit. The same argument can be used to prove that if the oriented adjacency between edges $i$ and $j$ is $+1$, then we have ${X}_{ij} = -2$ in the $N \to \infty$ limit.

\begin{center}
\includegraphics{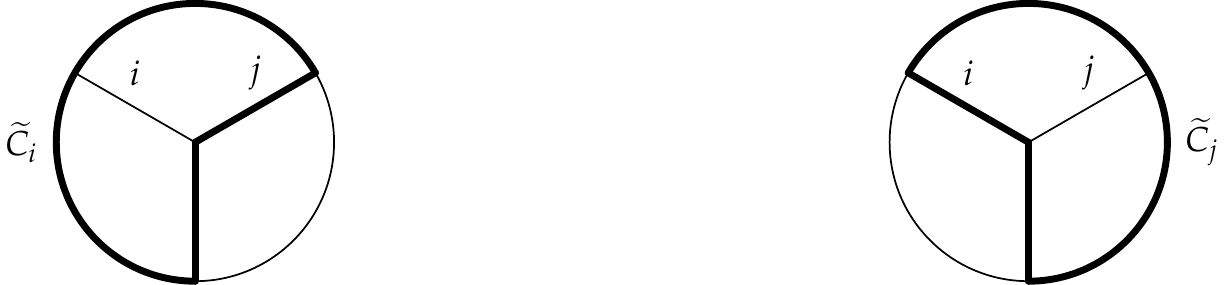}
\end{center}

The few additional cases which may arise include when the oriented adjacency is $\pm 2$ and when vertices, edges or faces in the diagrams above coincide. However, these may be handled in an entirely analogous manner which is not worthy of reproduction here.
\end{proof}

\begin{lemma} \label{matrix-rank}
The matrix $B$ has rank $6g-6+2n$.
\end{lemma}

\begin{proof}
We consider $\mathbb{R}^{E(\Gamma)}$ to be the real vector space with basis $\{E_1, E_2, \ldots, E_{6g-6+3n}\}$, the set of edges of $\Gamma$. Then the matrix $B$ represents the linear map $\mathbb{R}^{E(\Gamma)} \to \mathbb{R}^{E(\Gamma)}$ which describes the oriented adjacency between edges in $\Gamma$.

We also consider $\mathbb{R}^n$ to be the real vector space with basis $\{F_1, F_2, \ldots, F_n\}$, the set of faces in the cell decomposition corresponding to $\Gamma$. In Section~\ref{background}, we defined $A_\Gamma: \mathbb{R}^{E(\Gamma)} \to \mathbb{R}^n$ to be the linear map which represents adjacency between the faces and edges of $\Gamma$. The transpose $A^t: \mathbb{R}^n \to \mathbb{R}^{E(\Gamma)}$ is the linear map which sends a face to the sum of the edges adjacent to that face, counted with multiplicity. We will show that the composition of these two linear maps satisfies $B \circ A^t = 0$.

Suppose that $F$ is a face which is adjacent to the $m$ not necessarily distinct edges $E_1, E_2, \ldots, E_m$, as shown in the diagram below. Furthermore, suppose that the edges $E_k$ and $E_{k+1}$ are also adjacent to the edge $\overline{E}_k$, where the subscripts are taken modulo $m$. The following calculation shows that $B \circ A^t = 0$ holds on a basis for $\mathbb{R}^n$, so we have $\text{im}~A^t \subseteq \text{ker}~B$.

\begin{multicols}{2}
\begin{center}
\includegraphics{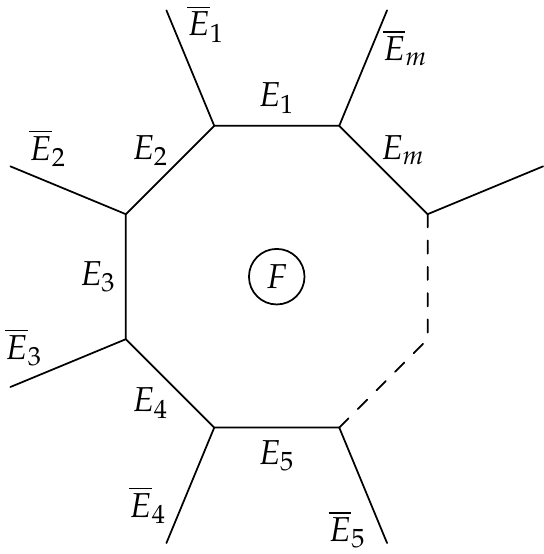}
\end{center}

\begin{align*}
&~B \circ A^t(F) \\
=&~B(E_1 + E_2 + \cdots + E_m) \\
=&~\sum B(E_k) \\
=&~\sum (E_{k-1} - \overline{E}_{k-1} + \overline{E}_{k} - E_{k+1}) \\
=&~\sum E_{k-1} - \sum \overline{E}_{k-1} + \sum \overline{E}_{k} - \sum E_{k+1} \\
=&~0
\end{align*}
\end{multicols}

Now if $\sum a_k E_k \in \ker B$, then $a_{i-1} + a_i - \overline{a}_{i-1} = a_i + a_{i+1} - \overline{a}_i$ for all $i$. So to the face $F$, we can associate the well-defined value
\[
b = \frac{a_1 + a_2 - \overline{a}_1}{2} = \frac{a_2 + a_3 - \overline{a}_2}{2} = \cdots = \frac{a_m + a_1 - \overline{a}_m}{2}.
\]
Since $B(\sum b_k F_k) = \sum a_k E_k$, we have $\ker B \subseteq \text{im}~A^t$. In particular, we have established that $\ker B = \text{im}~A^t$.

Now if $\sum b_k F_k \in \ker A^T$ and $F_i$ and $F_j$ are adjacent faces, then $b_i + b_j = 0$. So if the faces $F_i, F_j, F_k$ meet at a vertex, then $b_i + b_j = b_j + b_k = b_k + b_i = 0$, which implies that $b_i = b_j = b_k = 0$. Since $\Gamma$ is connected, we  deduce that $\text{ker}~A^t = 0$. Now invoke the rank--nullity theorem to conclude that $\dim(\text{im}~B) = 6g-6+2n$.
\end{proof}

The previous lemma asserts that we may relabel the edges of $\Gamma$ in such a way that the matrix $\widehat{B}$ formed by taking the first $6g-6+2n$ rows and $6g-6+2n$ columns of $B$ is invertible. The matrix $\widehat{X}$ is defined analogously from the matrix $X$.

We are now in a position to prove Theorem~\ref{equal-forms} which states that, in the $N \to \infty$ limit, the symplectic form $\frac{f^*\omega}{N^2}$ converges pointwise to a piecewise linear 2-form $\Omega_L$ on the locus of trivalent metric ribbon graphs in ${\mathcal MRG}_{g,n}(\x)$. Furthermore, this coincides with the piecewise linear 2-form $\Omega_K$ introduced by Kontsevich in his proof of the Witten--Kontsevich theorem \cite{kon}.

\begin{proof}[Proof of Theorem~\ref{equal-forms}]
We have shown that in the $N \to \infty$ limit, $X = -2B$ and $d\widehat{\ell}_k = -2de_k$. Now use Proposition~\ref{wolpert} to obtain
\[
\lim_{N \to \infty} \frac{f^*\omega}{N^2} = - \lim_{N \to \infty} \sum_{i < j}~[\widehat{X}^{-1}]_{ij}~d\widehat{\ell}_i \wedge d\widehat{\ell}_j = 2 \sum_{i < j}~[\widehat{B}^{-1}]_{ij}~de_i \wedge de_j.
\]
So for every trivalent ribbon graph $\Gamma$, the symplectic form $\frac{f^*\omega}{N^2}$ converges pointwise on ${\mathcal MRG}_\Gamma(\x)$. Furthermore, the limiting 2-form has constant coefficients with respect to the local coordinates $e_1, e_2, \ldots, e_{6g-6+3n}$. Therefore, the asymptotic Weil--Petersson form $\Omega_L$ is a piecewise linear 2-form on the locus of trivalent metric ribbon graphs in ${\mathcal MRG}_{g,n}(\x)$.

Kontsevich defines a piecewise linear 2-form $\Omega_K$ on the locus of trivalent metric ribbon graphs in ${\mathcal MRG}_{g,n}(\x)$ in the following way. For each face of $\Gamma$, choose one distinguished edge adjacent to that face. This allows us to turn the cyclic ordering of the edges around a face into a total ordering by declaring the distinguished edge to be first. Now define the matrix $K$ according to the following rule.
\[
K_{ij} = \sum_{\text{faces}} \left\{ \begin{array}{cl} +1 & \text{if edge $i$ comes before edge $j$} \\ -1 & \text{if edge $i$ comes after edge $j$} \\ 0 & \text{if edge $i$ or edge $j$ is not adjacent to the face} \end{array} \right.
\]
Since $K$ is skew-symmetric, it can be used to define the 2-form
\[
\Omega_K = \frac{1}{4} \sum_{i < j} K_{ij}~de_i \wedge de_j.
\]
One can prove that $\Omega_K$ is well-defined and non-degenerate on ${\mathcal MRG}_\Gamma(\x)$ for a trivalent ribbon graph $\Gamma$. We define a vector field on ${\mathcal MRG}_\Gamma(\x)$ corresponding to edge $i$ by
\[
T_i = \sum_j B_{ij}~\frac{\partial}{\partial e_j}.
\]
Then a straightforward computation shows that
\[
\iota_{T_i} \Omega_K = -2d e_i \qquad \text{and} \qquad \Omega_K(T_i, T_j) = -2B_{ij}.
\]

On the other hand, results of Wolpert \cite{wol1} assert that on ${\mathcal M}_{g,n}(N\x)$, the Fenchel--Nielsen coordinates satisfy
\[
\iota_{\tau_i} \omega = -d \ell_i \qquad \text{and} \qquad \omega(\tau_i, \tau_j) = X_{ij}.
\]
Using the fact that $X = -2B$ and $d\widehat{\ell}_k = -2de_k$ in the $N \to \infty$ limit, it follows that $\Omega_K$ and $\Omega_L$ are equal on the locus of trivalent metric ribbon graphs in ${\mathcal MRG}_{g,n}(\x)$.
\end{proof}

An alternative proof of this result appears in the work of Mondello \cite{mon2}. Among other differences, Mondello uses Penner coordinates to analyse the Weil--Petersson Poisson structure on Teichm\"{u}ller space and produces Theorem~\ref{equal-forms} as a byproduct. It seems that our choice of coordinates is more well-suited for the purpose of analysing the asymptotic behaviour of the Weil--Petersson symplectic form. We believe that it offers a more intuitive and less computational proof of Theorem~\ref{equal-forms}.

\section{Applications to intersection theory on moduli spaces of curves} \label{applications}

\subsection{Witten--Kontsevich theorem}

One of the landmark results concerning intersection theory on moduli spaces of curves is the Witten--Kontsevich theorem, which governs psi-class intersection numbers
\[
\int_{\M_{g,n}} \psi_1^{\alpha_1} \psi_2^{\alpha_2} \ldots \psi_n^{\alpha_n} \in \mathbb{Q}.
\]
In his foundational paper \cite{wit}, Witten conjectured that a particular generating function for these numbers is a tau function for the Korteweg--de Vries integrable hierarchy. This gives an effective recursion for calculating all psi-class intersection numbers. An equivalent formulation of the conjecture states that the generating function for psi-class intersection numbers satisfies certain Virasoro constraints.

Kontsevich's proof \cite{kon} of Witten's conjecture uses combinatorial polygon bundles over the combinatorial moduli space in order to represent the psi-classes. He writes down the 2-form $\Omega_K$ and determines the corresponding volume of the combinatorial moduli space. Taking the Laplace transform of the Kontsevich volume results in the combinatorial formula stated as Theorem~\ref{kcf}. This identity reduces the calculation of psi-class intersection numbers on $\M_{g,n}$ to a certain weighted enumeration of trivalent ribbon graphs of type $(g,n)$. From this point, Kontsevich introduced a matrix integral to handle the ribbon graph enumeration and the known relationship between matrix integrals and integrable hierarchies allowed him to deduce Witten's conjecture.

More recently, Mirzakhani \cite{mir1, mir2} has found an alternative proof of the Witten--Kontsevich theorem by calculating volumes of moduli spaces of hyperbolic surfaces with respect to the Weil--Petersson form.
\[
V_{g,n}(\LL) = \int_{{\mathcal M}_{g,n}(\LL)} \frac{\omega^{3g-3+n}}{(3g-3+n)!}
\]
She uses a generalisation of McShane's identity concerning the length spectrum of a hyperbolic surface to produce a method for integrating over moduli spaces of hyperbolic surfaces. In particular, she is able to write down a recursive formula for the Weil--Petersson volumes. On the other hand, Mirzakhani employs the method of symplectic reduction to generalise Proposition~\ref{de-rham} in the following way.

\begin{proposition} \label{mirzakhani}
The Weil--Petersson form $\omega$ on the moduli space of hyperbolic surfaces ${\mathcal M}_{g,n}(\LL)$ induces the de~Rham cohomology class
\[
[\omega] = 2\pi^2 \kappa_1 + \frac{1}{2}L_1^2 \psi_1 + \frac{1}{2}L_2^2 \psi_2 + \cdots + \frac{1}{2}L_n^2 \psi_n \in H^2(\M_{g,n}; \mathbb{R}).
\]
\end{proposition}

A direct corollary of this result is the fact that the Weil--Petersson volume of ${\mathcal M}_{g,n}(\LL)$ is given by the expression
\[
V_{g,n}(\LL) = \sum_{|\text{\boldmath$\alpha$}| + m = 3g-3+n} \frac{(2\pi^2)^m \int_{\M_{g,n}} \psi_1^{\alpha_1} \psi_2^{\alpha_2} \cdots \psi_n^{\alpha_n} \kappa_1^m}{2^{|\text{\boldmath$\alpha$}|} \alpha_1! \alpha_2! \cdots \alpha_n! m!} L_1^{2\alpha_1} L_2^{2\alpha_2} \cdots L_n^{2\alpha_n}.
\]
Combining her volume calculation with this result yields a recursive formula for the psi-class intersection numbers. Thus, Mirzakhani is able to deduce the Witten--Kontsevich theorem. Her proof is particularly striking since it directly verifies the Virasoro constraints, completely bypasses the theory of matrix integrals, and uses hyperbolic geometry in a fundamental way.

Theorem~\ref{kcf}, which we refer to as Kontsevich's combinatorial formula, states that for the moduli space of curves $\M_{g,n}$, we have the following equality of rational polynomials in $s_1, s_2, \ldots, s_n$.
\[
\sum_{|\text{\boldmath$\alpha$}| = 3g-3+n} \int_{\M_{g,n}} \psi_1^{\alpha_1} \psi_2^{\alpha_2} \cdots \psi_n^{\alpha_n} ~ \prod_{k=1}^n \frac{(2\alpha_k - 1)!!}{s_k^{2\alpha_k+1}} = \sum_{\Gamma} \frac{2^{2g-2+n}}{|\text{Aut}(\Gamma)|} \prod_{e \in E(\Gamma)} \frac{1}{s_{\ell(e)} + s_{r(e)}}
\]
Here, we use the notation $|\text{\boldmath$\alpha$}|$ as a shorthand for $\alpha_1 + \alpha_2 + \cdots + \alpha_n$. The sum on the right hand side is over the trivalent ribbon graphs of type $(g,n)$. We write $\text{Aut}(\Gamma)$ and $E(\Gamma)$ for the automorphism group and edge set of $\Gamma$, respectively. For an edge $e$, the expressions $\ell(e)$ and $r(e)$ denote the labels of the faces on its left and right.

\begin{proof}[Proof of Theorem~\ref{kcf}]
By Proposition~\ref{mirzakhani}, we can write the asymptotics of the Weil--Peterson volume as
\[
\lim_{N \to \infty} \frac{V_{g,n}(N\x)}{N^{6g-6+2n}} = \sum_{|\text{\boldmath$\alpha$}| = 3g-3+n} \frac{\int_{\M_{g,n}} \psi_1^{\alpha_1} \psi_2^{\alpha_2} \cdots \psi_n^{\alpha_n}}{2^{3g-3+n} \alpha_1! \alpha_2! \cdots \alpha_n!} x_1^{2\alpha_1} x_2^{2\alpha_2} \cdots x_n^{2\alpha_n}.
\]

We may alternatively express the asymptotics of the Weil--Peterson volume in the following way.
\begin{align*}
\lim_{N \to \infty} \frac{V_{g,n}(N\x)}{N^{6g-6+2n}} &= \frac{1}{(3g-3+n)!}\lim_{N \to \infty} \int_{{\mathcal M}_{g,n}(N\x)} \left( \frac{\omega}{N^2} \right)^{3g-3+n} \\
&= \frac{1}{(3g-3+n)!} \int_{{\mathcal MRG}_{g,n}(\x)} \left( \lim_{N \to \infty} \frac{f^*\omega}{N^2} \right)^{3g-3+n} \\
&= \int_{{\mathcal MRG}_{g,n}(\x)} \frac{\Omega_K^{3g-3+n}}{(3g-3+n)!} 
\end{align*}
To obtain the second line from the first, we invoke the Lebesgue dominated convergence theorem to move the limit inside the integral. And to obtain the third line from the second, we have used Theorem~\ref{equal-forms}.

After equating these two expressions and taking the Laplace transform, we arrive at the following equation.
\[
\sum_{|\text{\boldmath$\alpha$}| = 3g-3+n} \int_{\M_{g,n}} \psi_1^{\alpha_1} \psi_2^{\alpha_2} \cdots \psi_n^{\alpha_n} ~ \prod_{k=1}^n \frac{(2\alpha_k-1)!!}{s_k^{2\alpha_k+1}} = 2^{3g-3+n} ~ {\mathcal L} \left\{\int_{{\mathcal MRG}_{g,n}(\x)} \frac{\Omega_K^{3g-3+n}}{(3g-3+n)!} \right\}.
\]
The left hand side coincides with that of Kontsevich's combinatorial formula. The right hand side naturally splits as a sum, since the cell decomposition for ${\mathcal MRG}_{g,n}(\x)$ possesses one top-dimensional cell for each trivalent ribbon graph of type $(g,n)$. Using the edge lengths as local coordinates, we find that each top-dimensional cell is naturally a convex polytope inside which the volume form is constant. By explicitly performing the volume calculation, we obtain Kontsevich's combinatorial formula. We refer the reader to Kontsevich's original paper \cite{kon} for the precise details of the volume calculation.
\end{proof}

We have shown that the Kontsevich volumes arise as the asymptotics of the Weil--Petersson volumes. This provides a new proof of the Witten--Kontsevich theorem and makes explicit the connection between the work of Kontsevich \cite{kon} and Mirzakhani \cite{mir1, mir2}. The strength of our approach lies in the fact that it avoids the complications inherent in the groundbreaking work of Kontsevich --- namely, the non-standard compactification of the moduli space of curves and the justification of how psi-classes arise from integrating over the combinatorial moduli space.

\subsection{A recursive formula for Kontsevich volumes}

There is further mileage to be gained from the viewpoint that Kontsevich volumes arise as the asymptotics of the Weil--Petersson volumes. For example, one expects the existence of a recursive formula for the Kontsevich volumes whose proof is modelled on  Mirzakhani's calculation of the Weil--Petersson volumes --- indeed, such a recursion appears in \cite{ben-coc-saf-wos}. Since their viewpoint is similar to ours and their results complementary, it is apt to include their main theorem here for comparison. Consider the Kontsevich volume multiplied by the product of the perimeters
\[
W_{g,n}(\LL) = L_1 L_2 \cdots L_n \int_{\mathcal{MRG}_{g,n}(\LL)} \frac{\Omega_K^{3g-3+n}}{(3g-3+n)!}.
\]

\begin{theorem} \label{recursion} Let $S = \{1, 2, \ldots, n\}$ and for an index set $I = \{i_1, i_2, \ldots i_k\}$, let $\LL_I = (L_{i_1}, L_{i_2}, \ldots, L_{i_k})$. Then we have the following recursive formula for Kontsevich volumes.
\begin{align*}
& W_{g,n+1}(L_0, \LL_S) = \sum_{k=1}^n L_k \left[ \int_0^{L_0-L_k} (L_0-x) W_{g,n}(x, \LL_{S \setminus \{k\}})~dx + \int_{L_0-L_k}^{L_0+L_k} \frac{L_0+L_k-x}{2} W_{g,n}(x, \LL_{S \setminus \{k\}})~dx \right] \\
& + \iint_{0 \leq x+y \leq L_0} \frac{L_0-x-y}{2} \left[ W_{g-1,n+1}(x, y, \LL_S) + \sum_{g_1 + g_2 = g} \sum_{I_1 \sqcup I_2 = S} W_{g_1,|I_1|+1}(x, \LL_{I_1}) W_{g_2,|I_2|+1}(y, \LL_{I_2}) \right] dx~dy
\end{align*}
\end{theorem}

As expected, this equation bears a striking resemblance to Mirzakhani's recursive formula for Weil--Petersson volumes \cite{mir1}. It can be used to compute all of the Kontsevich volumes from the base cases
\[
W_{0,3}(L_1, L_2, L_3) = L_1L_2L_3 \qquad \text{and} \qquad W_{1,1}(L_1) = \frac{1}{48}L_1^3.
\]
Furthermore, if one considers the differential version of this formula, then the result is precisely the Virasoro constraint for the Witten--Kontsevich theorem. This provides yet another path to the Witten--Kontsevich theorem. Finally, we remark that the Laplace transform of this recursive formula is the topological recursion of Eynard and Orantin \cite{eyn-ora} applied to the spectral curve $x = \frac{1}{2}y^2$.

\subsection{Combinatorial cycles}

Witten introduced cycles in the combinatorial moduli space in the following way. Consider the closure of the subset of ${\mathcal MRG}_{g,n}(\LL)$ consisting of metric ribbon graphs with one vertex of degree $2k+3$. This defines a homology class $\overline{W}_k(\LL) \in H_{2k}({\mathcal MRG}_{g,n}(\LL); \mathbb{Q})$ which in turn defines a cohomology class $W_k \in H^{2k}(\M_{g,n}; \mathbb{Q})$ via Poincar\'{e} duality. In particular, this cohomology class is independent of the initial choice of $\LL$. Subsequently, Kontsevich introduced more general cycles $\overline{W}_{\mathbf{m}}(\LL) \in H_*({\mathcal MRG}_{g,n}(\LL); \mathbb{Q})$, where $\mathbf{m} = (m_0, m_1, m_2, \ldots)$ is a sequence of non-negative integers. These arise by taking the closure of the subset of ${\mathcal MRG}_{g,n}(\LL)$ consisting of metric ribbon graphs with $m_k$ vertices of degree $2k+3$ for each $k$. These are often referred to as combinatorial cycles and they define cohomology classes $W_{\mathbf{m}} \in H^*(\M_{g,n}; \mathbb{Q})$ via Poincar\'{e} duality.

As above, we would like to perform a volume calculation with respect to the asymptotic Weil--Petersson form, but now over a combinatorial cycle rather than the whole combinatorial moduli space. For example, one can integrate over the cycle $\overline{W}_{\mathbf{m}}(N\x) \subseteq {\mathcal MRG}_{g,n}(N\x)$ of codimension $2d$ and obtain
\[
\lim_{N \to \infty} \frac{1}{N^{6g-6+2n-2d}} \int_{\overline{W}_{\mathbf{m}}(N\x)} \frac{\omega^{3g-3+n-d}}{(3g-3+n-d)!}.
\]
By Proposition~\ref{mirzakhani}, this expression is equal to
\[
\sum_{|\text{\boldmath$\alpha$}| = 3g-3+n-d} \int_{\M_{g,n}} W_\mathbf{m} \psi_1^{\alpha_1} \psi_2^{\alpha_2} \cdots \psi_n^{\alpha_n} \frac{x_1^{2\alpha_1} x_2^{2\alpha_2} \cdots x_n^{2\alpha_n}}{2^{3g-3+n-d} \alpha_1! \alpha_2! \cdots \alpha_n!}.
\]

On the other hand, one may compute the volume explicitly by calculating the asymptotic Weil--Petersson form on the combinatorial cycle. One way to do this is to choose $6g-6+2n$ geodesic multicurves whose lengths locally parametrise the moduli space. By Lemma~\ref{length-lemma}, these lengths converge to a linear combination of the edge lengths in the metric ribbon graph in the $N \to \infty$ limit. In the trivalent case, we observed earlier that all angles between closed geodesics converge to 0 or $\pi$. In the case of higher degree vertices, the limiting angles may still be computed, using the observation that the vertex polygons converge to regular ideal polygons. Armed with this information, one may then invoke Proposition~\ref{wolpert} to write down an explicit formula for the asymptotic Weil--Petersson form on the combinatorial cycle. Following the proof of Kontsevich's combinatorial formula, one then expects the volume calculation to split naturally as a sum over the ribbon graphs of type $(g,n)$ with $m_k$ vertices of degree $2k+3$ for each $k$. Equating these two volume calculations with respect to the asymptotic Weil--Petersson form over the Witten cycle then yields an identity which relates intersection numbers on $\M_{g,n}$ to a certain weighted enumeration of ribbon graphs. In fact, this provides a method for computing intersection numbers of the form
\[
\int_{\M_{g,n}} W \psi_1^{\alpha_1} \psi_2^{\alpha_2} \cdots \psi_n^{\alpha_n},
\]
where $W \in H^*(\M_{g,n}; \mathbb{Q})$ is the Poincar\'{e} dual to a combinatorial cycle.

\begin{example} \label{combinatorial-cycle}
Consider the Witten cycle $\overline{W}_1(Nx_1, Nx_2) \subseteq {\mathcal MRG}_{1,2}(Nx_1, Nx_2)$ defined by the set of metric ribbon graphs with at least one vertex of degree at least five. There are eight ribbon graphs of type $(1,2)$ with one vertex of degree five and one vertex of degree three --- two for each of the diagrams below corresponding to the two ways to label the faces. Note that each of these ribbon graphs has trivial automorphism group.

\begin{center}
\includegraphics{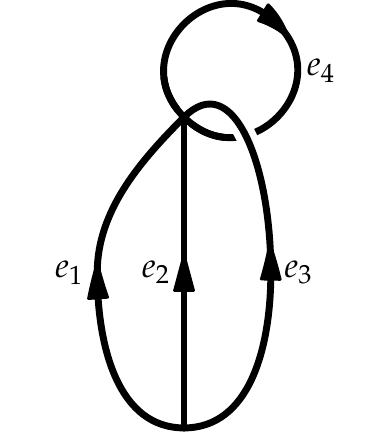}
\includegraphics{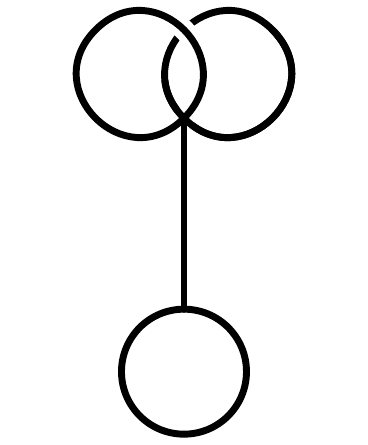}
\includegraphics{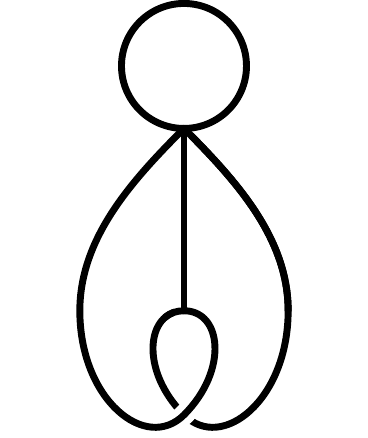}
\includegraphics{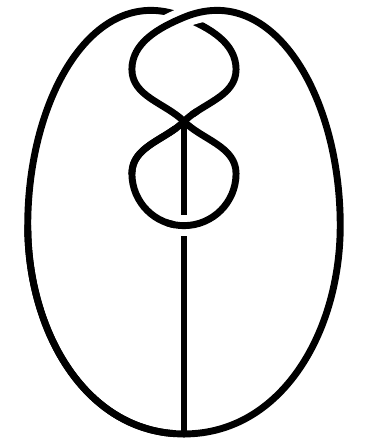}
\end{center}

We are interested in calculating the Laplace transform of the volume
\[
{\mathcal L} \left\{ \lim_{N \to \infty} \frac{1}{N^2} \int_{\overline{W}_1(Nx_1, Nx_2)} \omega \right\} = {\mathcal L} \left\{\int_{\overline{W}_1(x_1, x_2)} \Omega \right\},
\]
where $\Omega$ represents the asymptotic Weil--Petersson form over the Witten cycle $\overline{W}_1(x_1, x_2)$. Each of the eight metric ribbon graphs defines a top-dimensional cell in $\overline{W}_1(x_1, x_2)$, so we obtain a contribution to the volume from each. For example, let us explicitly determine the contribution from the leftmost ribbon graph $\Gamma$ above, where the face with two edges along its boundary is labelled 1 while the face with six edges along its boundary is labelled 2. In order to write down the asymptotic Weil--Petersson form, we look for four simple closed multicurves on $\Gamma$ whose geodesic representatives have lengths which locally parametrise Teichm\"{u}ller space. For example, one may choose the following, which have been described as a cyclic sequence of oriented edges.
\[
\begin{array}{rcl}
C_1 &=& [e_2, e_4, \overline{e}_2, e_3, \overline{e}_4, \overline{e}_3] \\
C_2 &=& [e_1, e_4, \overline{e}_1, e_3, \overline{e}_4, \overline{e}_3] \\
C_3 &=& [e_1, e_4, \overline{e}_2] \cup [e_4] \\
C_4 &=& [e_1, e_4, \overline{e}_2, e_3, \overline{e}_4, \overline{e}_1, e_2, \overline{e}_3]
\end{array}
\]

Their normalised lengths $\widehat{\ell}_k = \frac{\ell_k}{N}$ satisfy the following equations.
\[
\begin{array}{rcl}
\widehat{\ell}_1 &=& 2e_2 + 2e_3 + 2e_4 = x_1 + x_2 - 2e_1 \\
\widehat{\ell}_2 &=& 2e_1 + 2e_3 + 2e_4 = x_1 + x_2 - 2e_2 \\
\widehat{\ell}_3 &=& e_1 + e_2 + 2e_4 = x_2 - 2e_3 \\
\widehat{\ell}_4 &=& 2e_1 + 2e_2 + 2e_3 + 2e_4 = x_1 + x_2
\end{array}
\qquad \Rightarrow \qquad
\begin{array}{rcl}
d\widehat{\ell}_1 &=& -2de_1 \\
d\widehat{\ell}_2 &=& -2de_2 \\
d\widehat{\ell}_3 &=& -2de_3 \\
d\widehat{\ell}_4 &=& 0
\end{array}
\]

In order to use Wolpert's expression for the Weil--Petersson form, it is also necessary to calculate the angles between the four multicurves in the $N \to \infty$ limit. We do this using the observation that the vertex polygons converge to regular ideal polygons. For example, when two curves intersect at a degree five vertex of $\Gamma$, the angle between the corresponding curves in $S(N\Gamma)$ converges to the angle $\theta$ indicated in the diagram below, where $P_1P_2P_3P_4P_5$ is an ideal regular pentagon. Using standard hyperbolic trigonometry, one can calculate that $\cos \theta = \sqrt{5} - 2$.

\begin{center}
\includegraphics{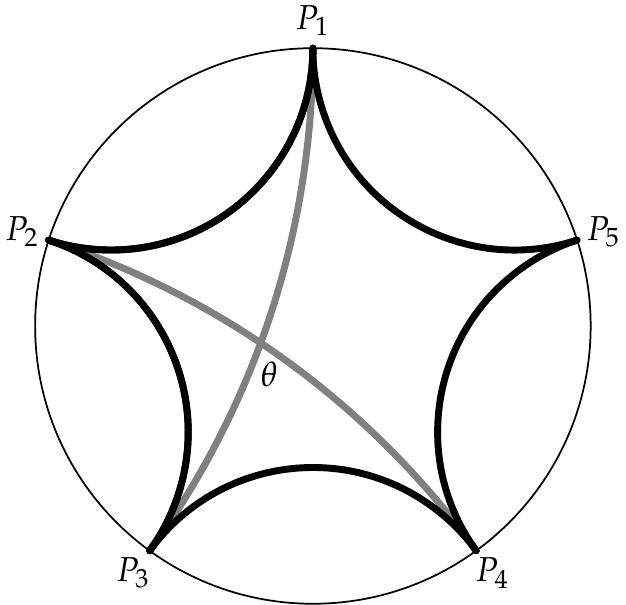}
\end{center}

We may now invoke Proposition~\ref{wolpert} which requires us to calculate the following two matrices.
\[
X = \left[ \begin{array}{cccc} 0 & \sqrt{5}-1 & -2 & -2 \\
1-\sqrt{5} & 0 & 2 & \sqrt{5}-1 \\
2 & -2 & 0 & 1-\sqrt{5} \\
2 & 1-\sqrt{5} & \sqrt{5}-1 & 0 \end{array} \right]
\;\; \Rightarrow \;\;
X^{-1} = \frac{1}{8} \left[ \begin{array}{cccc} 0 & -1-\sqrt{5} & -1-\sqrt{5} & 3+\sqrt{5} \\
1+\sqrt{5} & 0 & -3-\sqrt{5} & 3+\sqrt{5} \\
1+\sqrt{5} & 3+\sqrt{5} & 0 & 1+\sqrt{5} \\
-3-\sqrt{5} & 3-\sqrt{5} & -1-\sqrt{5} & 0 \end{array} \right]
\]

The asymptotic Weil--Petersson form over the Witten cycle $\overline{W}_1(x_1, x_2)$ can now be expressed as follows.
\[
\Omega = - \sum_{i < j}~[X^{-1}]_{ij}~d\widehat{\ell}_i \wedge d\widehat{\ell}_j = \frac{1+\sqrt{5}}{2} de_1 \wedge de_2 + \frac{1+\sqrt{5}}{2} de_1 \wedge de_3 + \frac{3+\sqrt{5}}{2} de_2 \wedge de_3 = de_2 \wedge de_3
\]
Here, we have used the fact that $de_2 = -de_1$, since $e_1 + e_2 = x_1$ is a constant.

After integrating over the convex polytope defined by $\Gamma$, one obtains
\[
{\mathcal L} \left\{ \int_{A_\Gamma \mathbf{e} = \x} de_2 \wedge de_3 \right\} = \frac{1}{2(s_1+s_2)^2 s_2^2}.
\]
Here, the matrix $A_\Gamma$ is the adjacency between faces and edges in the ribbon graph $\Gamma$. Furthermore, $s_1$ and $s_2$ are the Laplace transform variables of $x_1$ and $x_2$, respectively. For each of the eight ribbon graphs above, one obtains a similar contribution to the Laplace transform of the volume. The sum of the eight contributions simplifies rather remarkably in the following way.
\[
\begin{array}{cccccccc}
 & \frac{1}{2(s_1+s_2)^2 s_2^2} & + & \frac{1}{4(s_1+s_2) s_2^3} & + & \frac{1}{4(s_1+s_2) s_2^3} & + & \frac{1}{(s_1+s_2)^3 s_2} \\
+ & \frac{1}{2(s_1+s_2)^2 s_1^2} & + & \frac{1}{4(s_1+s_2) s_1^3} & + & \frac{1}{4(s_1+s_2) s_1^3} & + & \frac{1}{(s_1+s_2)^3 s_1}
\end{array}
= \frac{1}{2s_1s_2^3} + \frac{1}{2s_1^3s_2}
\]

Comparing this calculation to
\[
\lim_{N \to \infty} \frac{1}{N^2} \int_{W_1(N\x)} \omega = \frac{x_1^2}{2} \int_{\M_{1,2}} W_1 \psi_1 + \frac{x_2^2}{2} \int_{\M_{1,2}} W_1 \psi_2,
\]
one obtains
\[
\int_{\M_{1,2}} W_1 \psi_1 = \int_{\M_{1,2}} W_1 \psi_2 = 1.
\]
This result is easily verified by using $W_1 = 12 \kappa_1$, a fact first proved by Penner \cite{pen}.
\end{example}

\bibliographystyle{hacm}
\bibliography{asymptotic-WP}

\begin{thebibliography}{10}

\bibitem{ben-coc-saf-wos}
{\sc Bennett, J., Cochran, D., Safnuk, B., and Woskoff, K.}
\newblock Topological recursion for symplectic volumes of moduli spaces of
  curves.
\newblock \href{http://arxiv.org/abs/1010.1747v1}{arXiv:1010.1747v1 [math.SG]}.

\bibitem{bow-eps}
{\sc Bowditch, B. H. \and~Epstein, D. B.~A.}
\newblock Natural triangulations associated to a surface.
\newblock {\em Topology 27}, 1 (1988), 91--117.

\bibitem{bus}
{\sc Buser, P.}
\newblock {\em Geometry and spectra of compact {R}iemann surfaces}, vol.~106 of
  {\em Progress in Mathematics}.
\newblock Birkh\"auser Boston Inc., Boston, MA, 1992.

\bibitem{do}
{\sc Do, N.}
\newblock {\em Intersection theory on moduli spaces of curves via hyperbolic
  geometry}.
\newblock PhD thesis, The University of Melbourne, 2008.

\bibitem{eyn-ora}
{\sc Eynard, B. \and~Orantin, N.}
\newblock Invariants of algebraic curves and topological expansion.
\newblock {\em Commun. Number Theory Phys. 1}, 2 (2007), 347--452.

\bibitem{kon}
{\sc Kontsevich, M.}
\newblock Intersection theory on the moduli space of curves and the matrix
  {A}iry function.
\newblock {\em Comm. Math. Phys. 147}, 1 (1992), 1--23.

\bibitem{mir1}
{\sc Mirzakhani, M.}
\newblock Simple geodesics and {W}eil-{P}etersson volumes of moduli spaces of
  bordered {R}iemann surfaces.
\newblock {\em Invent. Math. 167}, 1 (2007), 179--222.

\bibitem{mir2}
{\sc Mirzakhani, M.}
\newblock {W}eil--{P}etersson volumes and intersection theory on the moduli
  space of curves.
\newblock {\em J. Amer. Math. Soc. 20}, 1 (2007), 1--23.

\bibitem{mon2}
{\sc Mondello, G.}
\newblock Triangulated {R}iemann surfaces with boundary and the
  {W}eil--{P}etersson {P}oisson structure.
\newblock {\em J. Differential Geom. 81}, 2 (2009), 391--436.

\bibitem{pen}
{\sc Penner, R.~C.}
\newblock The {P}oincar\'e dual of the {W}eil--{P}etersson {K}\"ahler two-form.
\newblock In {\em Perspectives in mathematical physics}, Conf. Proc. Lecture
  Notes Math. Phys., III. Int. Press, Cambridge, MA, 1994, pp.~229--249.

\bibitem{wit}
{\sc Witten, E.}
\newblock Two-dimensional gravity and intersection theory on moduli space.
\newblock In {\em Surveys in differential geometry ({C}ambridge, {MA}, 1990)}.
  Lehigh Univ., Bethlehem, PA, 1991, pp.~243--310.

\bibitem{wol2}
{\sc Wolpert, S.}
\newblock On the homology of the moduli space of stable curves.
\newblock {\em Ann. of Math. (2) 118}, 3 (1983), 491--523.

\bibitem{wol1}
{\sc Wolpert, S.}
\newblock On the symplectic geometry of deformations of a hyperbolic surface.
\newblock {\em Ann. of Math. (2) 117}, 2 (1983), 207--234.

\bibitem{wol3}
{\sc Wolpert, S.~A.}
\newblock Convexity of geodesic-length functions: a reprise.
\newblock In {\em Spaces of {K}leinian groups}, vol.~329 of {\em London Math.
  Soc. Lecture Note Ser.} Cambridge Univ. Press, Cambridge, 2006, pp.~233--245.

\end{thebibliography}

\end{document}